\documentclass[12pt]{article}

\usepackage{verbatim}
\usepackage[T1]{fontenc}
\usepackage{setspace}
\usepackage{indentfirst}
\usepackage{multicol}
\usepackage[margin = 1in]{geometry}
\usepackage[hidelinks]{hyperref}

\usepackage{amsthm}
\usepackage{amsmath}
\usepackage{amsfonts}
\usepackage{amssymb}
\usepackage{mathdots}
\usepackage{mathtools}
\usepackage{enumitem}
\setlist[enumerate]{nosep}
\usepackage{xcolor}
\usepackage{pdflscape}
\usepackage{mathrsfs}
\usepackage{titling}
\usepackage{times}
\usepackage{bm}
\usepackage{bbold}

\usepackage{pgfplots}
\pgfplotsset{compat=1.14}
\usepackage{tikz-cd}
\usepackage{tikz}
\usetikzlibrary{matrix,arrows,decorations.pathmorphing,mindmap,fit}
\usetikzlibrary{shapes.geometric}
\usepackage[usestackEOL]{stackengine}
\usepackage{fancybox}
\usepackage{caption}
\usepackage{nicematrix}
\usepackage{dynkin-diagrams}
\usepackage{graphicx}

\usepackage{etoolbox}

\makeatletter
\newcounter{author}
\renewcommand*\author[1]{%
  \stepcounter{author}%
  \ifnum\c@author=1
    \gdef\@author{#1}%
  \else
    \xdef\@author{\unexpanded\expandafter{\@author\and#1}}%
  \fi
  \csgdef{author@\the\c@author}{#1}}
\newcommand*\email[1]{%
  \csgdef{email@\the\c@author}{#1}}
\newcommand*\address[1]{%
  \csgdef{address@\the\c@author}{#1}}
\AtEndDocument{%
  \xdef\author@count{\the\c@author}%
  \c@author=1
  \print@authors}
\newcommand*\print@authors{%
  \ifnum\c@author>\author@count
  \else
    \print@author{\the\c@author}%
    \advance\c@author by 1
    \expandafter\print@authors
  \fi}
\newcommand*\print@author[1]{%
  \par\medskip
  \begin{tabular}{@{}l@{}}%
    \textsc{Addresses of \csuse{author@#1}}\\
    \csuse{address@#1}\\
    \textit{E-mail address}:
    \href{mailto:\csuse{email@#1}}{\texttt{\csuse{email@#1}}}
  \end{tabular}}
\makeatother

\usepackage[toc,page]{appendix}
\usepackage{todonotes}

\DeclareMathOperator{\Ad}{Ad}

\newcommand{\cmapsto}{\stackrel{c}{\longmapsto}}
\newcommand{\taumapsto}{\stackrel{\tau}{\longmapsto}}

\newcommand{\ZZ}{\mathbb{Z}}

\newcommand{\RR}{\mathbb{R}}
\newcommand{\R}{\mathbb{R}}
\newcommand{\CC}{\mathbb{C}}
\newcommand{\C}{\mathbb{C}}

\newcommand{\Omin}{\mathcal{O}_{\textrm{min}}}
\newcommand{\Onilp}{\mathcal{O}_{\textrm{nilp}}}
\newcommand{\Ominnn}{\Ominbar\cap(\Lien^+\oplus\Lien^-)}
\newcommand{\OminnnAtwo}{\Ominbar^{A_2}\cap(\Lien^+\oplus\Lien^-)}
\newcommand{\Ominn}{\Ominbar\cap\Lien^+}
\newcommand{\Onilpn}{\Onilpbar\cap\Lien^+}

\newcommand{\Ominbar}{\overline{\mathcal{O}}_\textrm{min}}
\newcommand{\Onilpbar}{\overline{\mathcal{O}}_\textrm{nilp}}

\newcommand{\slc}{\mathfrak{sl}_3}

\newcommand{\slnn}{\mathfrak{sl}_{n+1}}

\newcommand{\Lien}{\mathfrak{n}}

\newcommand{\Lieh}{\mathfrak{h}}
\newcommand{\init}{\mathrm{in}}

\usepackage{theoremref}
\newtheorem{Thm}{Theorem}[section]

\newtheorem{Lem}[Thm]{Lemma}
\newtheorem{Cor}[Thm]{Corollary}

\theoremstyle{definition}
\newtheorem{Def}[Thm]{Definition}
\newtheorem{Ex}[Thm]{Example}

\theoremstyle{remark}

\newtheorem{rem}[Thm]{Remark}

\newtheoremstyle{case}{}{}{}{}{}{:}{ }{}
\theoremstyle{case}

\def \CC {\mathbb{C}}

\usepackage{tikzsymbols}
\DeclareSymbolFont{extraup}{U}{zavm}{m}{n}
\DeclareMathSymbol{\varheart}{\mathalpha}{extraup}{86}

\begin{document}
\title{Minimal Nilpotent Orbits and Toric Varieties}
\author{Boming Jia}
\address{Yau Mathematical Sciences Center \\ Jingzhai 301, Tsinghua University \\ Beijing, 100084, China}
\email{jiabm@tsinghua.edu.cn}

\author{Yu Li}
\address{Department of Mathematics \\ University of Notre Dame \\ 255 Hurley Bldg. \\ Notre Dame, IN 46556, USA}
\email{liyu9112@gmail.com}
\date{}

\setlength{\droptitle}{-5em}
\maketitle

\vspace{-5em}
\begin{abstract}
Let $\Ominnn$ be the collection of elements of $\mathfrak{sl}_{n+1}(\CC)$ with rank less than or equal to $1$ and with all diagonal entries equal to zero. We show that the coordinate ring $\C[\Ominnn]$ of the scheme-theoretic intersection $\Ominnn$ has a flat degeneration to the ring of $(\CC^{\times})^n$-equivariant cohomology of the projective toric variety associated with the fan of compatible subsets of almost positive roots 
%$\Phi_{\geq-1}$
of type $C_n$. Then we compute the Hilbert series of $\C[\Ominnn]$ and prove that $\Ominnn$ is reduced and Gorenstein. Moreover, our proof method allows us to prove that the scheme-theoretic intersection $\Ominn$, of which the irreducible components are known as the ``orbital varieties'', is reduced and Cohen-Macaulay.
\end{abstract}

%\tableofcontents

\section{Introduction}
%\hspace{3em}{\Huge\textbf{Pure $TM$ Tree New Bee}}\\
Let $\mathfrak{sl}_{n+1}(\CC)$ be the Lie algebra consisting of traceless $(n+1)$-by-$(n+1)$ matrices with complex entries.
The minimal (nonzero) nilpotent adjoint orbit $\Omin$ consists of rank-one (hence nilpotent) elements of $\mathfrak{sl}_{n+1}(\CC)$, and its closure $\Ominbar=\Omin\cup\{0\}.$

In the context of symplectic duality, Hikita's conjecture \cite{Hikita} predicts that for a pair
\[Y \longrightarrow X \quad \text{and} \quad Y^{\vee} \longrightarrow X^{\vee}
\]
of dual symplectic resolutions with Hamiltonian torus actions, there exists a ring isomorphism $\CC[X^{\C^\times}] \cong H^*(Y^{\vee})$ between the coordinate ring of the scheme-theoretic fixed points $X^{\C^\times}$ (with respect to a given generic embedding of $\C^\times$ into the torus that acts Hamiltonianly on $Y \rightarrow X$) and the cohomology ring of $Y^{\vee}$.  For the pair
\[
T^*\mathbb P^n \longrightarrow \Ominbar \quad \text{and} \quad \widetilde{\CC^2/\ZZ}_{n+1} \longrightarrow \CC^2/\ZZ_{n+1},
\]
where $\widetilde{\CC^2/\ZZ}_{n+1}$ is the minimal resolution of the Kleinian singularity $\CC^2/\ZZ_{n+1}$, Hikita's conjecture has recently been confirmed by P. Shlykov \cite{OminHikita}. Let $H<SL_{n+1}(\CC)$ be the maximal torus consisting of diagonal elements, and $\mathfrak{h}$ its Lie algebra.  The adjoint action of $H$ on $\Ominbar$ is Hamiltonian. Fix a generic $\C^\times \hookrightarrow H$. Then the scheme theoretic $\C^\times$-fixed points $\Ominbar^{\C^\times}$ is isomorphic to the scheme theoretic intersection $\Ominbar\cap\Lieh$, and Shlykov's theorem states:
\begin{Thm}[\cite{OminHikita}]\label{thm:Hikita}
    The coordinate ring of the scheme $\Ominbar\cap\Lieh$ is isomorphic to the cohomology ring of the minimal resolution of the Kleinian singularity, i.e.,
    $$
            \CC[\Ominbar\cap\Lieh]\cong H^*(\widetilde{\CC^2/\ZZ}_{n+1}).
    $$
\end{Thm}
Let $\mu_H:{\mathcal O}_{\rm min}\rightarrow\mathfrak{h}^*$ denote the moment map under the adjoint action by $H$. Then the zero fiber
$$
    \mu_H^{-1}(0)=\mathcal O_{\rm min} \cap (\mathfrak n^+ \oplus \mathfrak n^-),
$$
where $\Lien^+$ (resp. $\Lien^-$) is the strictly upper (resp. lower) triangular Lie subalgebra of $\mathfrak{sl}_{n+1}(\CC)$. 

In this paper we prove a complementary result to Theorem \ref{thm:Hikita}. 
Let $\Sigma$ be the fan of compatible subsets of the set of almost positive roots $\Phi_{\geq-1}$ of type $C_n$. Let $X_\Sigma$ be the projective toric variety associated with the fan $\Sigma$.
\begin{Thm}\label{thm:mainThm} The coordinate ring $\CC[\Ominnn]$ degenerates flatly to the $(\CC^{\times})^n$-equivariant cohomology $H^*_{(\CC^{\times})^n}(X_{\Sigma})$, i.e., %there exists a flat family over the curve $\CC$ whose fiber over $z \in \CC^{\times} \subset \CC$ is $\Ominnn$ and whose fiber over $0 \in \CC$ is $\Spec H^*_{(\CC^{\times})^n}(X_{\Sigma})$.
there exists a flat morphism $\pi: \mathfrak{X} \rightarrow \CC$ such that
\begin{center}
    \begin{tikzcd}
	    {{\rm Spec} H^*_{(\CC^{\times})^n}(X_{\Sigma}) \cong \pi^{-1}(0)} & {\mathfrak{X}} & {\pi^{-1}(\CC^{\times}) \cong \left( \Ominnn \right) \times \CC^{\times}} \\
	    0 & {\CC} & {\CC^{\times}}
	    \arrow[hook, from=1-1, to=1-2]
	    \arrow[from=1-1, to=2-1]
	    \arrow["\pi", from=1-2, to=2-2]
	    \arrow[hook', from=1-3, to=1-2]
	    \arrow["{{\rm pr}_2}", from=1-3, to=2-3]
	    \arrow[hook, from=2-1, to=2-2]
	    \arrow[hook', from=2-3, to=2-2]
    \end{tikzcd}
\end{center}
\end{Thm}

Theorem \ref{thm:mainThm} is complementary to Theorem \ref{thm:Hikita} in that, instead of taking the intersection of $\Ominbar$ with $\Lieh$, we take the intersection with the complement $\Lien^+ \oplus \Lien^-$ of $\Lieh$ in $\slnn(\CC) = \Lien^+ \oplus \Lieh \oplus \Lien^-$.  Theorem \ref{thm:mainThm} also has the special feature that it relates the root systems of types $A$ and $C$ through a mechanism which is not a ``folding'' of a simply laced Dynkin diagram.  In fact, the type $A$ and $C$ root systems that appear in Theorem \ref{thm:mainThm} are of the same rank $n$, while folding the Dynkin diagram of type $A_{2n-1}$ gives the Dynkin diagram of type $C_n$.
As an application of Theorem \ref{thm:mainThm}, we prove:
\begin{Thm} \label{thm:HilbertSeries}
    The scheme-theoretic intersection $\Ominnn$ is reduced and Gorenstein.  Moreover, its Hilbert series with respect to the action of $\CC^{\times}$ that scales the matrix entries of $\slnn(\CC)$ is given by
    $$
        h_{\CC[\Ominnn]}(t)=\frac{\displaystyle\sum_{i=0}^{n}\binom{n}{i}^2t^i}{(1-t)^{n}}.
    $$
\end{Thm}

%The fan $\Sigma$ plays a crucial role in the study of the cluster algebra $\mathcal A^{C_n}$ of finite type $C_n$ and the combinatorics related to it.  For example, the rays (resp. maximal cones) of $\Sigma$ are in natural bijection with the mutable cluster variables (resp. seeds) of $\mathcal A^{C_n}$, and mutation of seeds corresponds to moving from one maximal cone to another that shares a face of codimension one.  The number of maximal cones of $\Sigma$, namely $\sum_{i=0}^{n}\binom{n}{i}^2$, is a type $C_n$ analogue of the Catalan numbers.  Moreover, the numbers $\binom{n}{i}^2$ that appear in the Hilbert series of $\Ominnn$ are type $C_n$ analogues of the Narayana numbers.  By Theorems \ref{thm:mainThm} and \ref{thm:HilbertSeries}, geometric properties of $\Ominnn$ have implications for the study of the cluster algebra $\mathcal A^{C_n}$ and Catalan combinatorics of type $C_n$.  These aspects will be dealt with in subsequent publications.

As a consequence of Theorem \ref{thm:HilbertSeries}, we prove that $\Ominbar$ is Gorenstein (see Corollary \ref{Cor:OminGor}).  Using the perspective of symplectic singularities, Beauville \cite[Proposition 1.3]{Beauville} proved that $\Ominbar$ is rational Gorenstein.  Our method, however, has the advantage of being elementary.

The numbers $\binom{n}{0}^2, \binom{n}{1}^2, \ldots, \binom{n}{n}^2$ that appear in the Hilbert series of $\CC[\Ominnn]$ are also known as the Narayana numbers of type $C_n$, and their sum $\sum_{i=0}^n \binom{n}{i}^2$ is known as the Catalan number of type $C_n$ \cite{FR}.  Theorem \ref{thm:HilbertSeries} thus establishes a connection between type $C$ Catalan combinatorics and the geometry of $\Ominbar$.

Orbital varieties are, by definition, irreducible components of the scheme-theoretic intersection $\Onilpn$, where $\Onilp$ is an arbitrary nilpotent adjoint orbit.  They are fundamental objects studied in geometric representation theory, serving as a bridge, via the Springer correspondence, between algebraic data such as Weyl group representations and geometric properties of the flag variety.  The geometry of orbital varieties is a subtle subject.  For example, to the best of our knowledge, it is not known when an arbitrary orbital variety is reduced.  Using the same technique as in the proof of Theorem \ref{thm:HilbertSeries}, we prove:
\begin{Thm} \label{thm:orbitalVar}
    The scheme-theoretic intersection $\Ominn$ is reduced and Cohen-Macaulay.
\end{Thm}

This paper is organized as follows. In Section 2, we recall basic definitions and properties of toric varieties and Stanley-Reisner rings. In Section 3, we study the case of $\mathfrak{sl}_3(\mathbb{C})$. By exhibiting a bijection between the almost positive roots of type $C_2$ and the roots of type $A_2$, we prove that $\C[\Ominnn]$ not only degenerates to $H^*_{(\CC^{\times})^2}(X_\Sigma)$ but is in fact isomorphic to it. In Section 4, we generalize this bijection to type $C_n$ and give explicit description of compatibility of almost positive roots. In Section 5, we recall basic definitions in computational algebraic geometry and prove a Gr\"obner basis result for matrices with prescribed zero entries. In Section 6, we apply the Gr\"obner basis result to prove our main theorem on flat degeneration and compute the Hilbert series of $\C[\Ominnn]$. The statements that $\Ominnn$ is reduced and Gorenstein, and $\Ominn$ is reduced and Cohen-Macaulay, are proved in the same section.

{\bf Acknowledgments} We would like to thank Joel Kamnitzer, Peng Shan, and Pavel Shlykov for stimulating discussions and useful suggestions.  Boming Jia was supported by NSFC Grant No. 12225108 and the Shuimu Scholar Program in Tsinghua University.

\section{Toric Varieties and Stanley-Reisner Rings}
We first recall several basic definitions and facts in toric geometry.
\begin{Def}
A \emph{toric variety} is an irreducible normal variety $X$ containing a torus $(\CC^\times)^n$ as an open dense subset, such that the action of the torus on itself by left multiplication extends to an action on $X$.
\end{Def}

\begin{Def} \label{def:cone}
    Let $L$ be a lattice in $\RR^n$, i.e., free $\ZZ$-submodule such that $L \otimes_{\ZZ} \RR = \RR^n$.  A \emph{strongly convex rational polyhedral cone} in $\RR^n$ is a subset of the form
\[
\sigma = \{\sum_{v \in V} c_v v \colon c_v \geq 0\},
\]
where $V \subset L$ is finite, such that $\sigma$ does not contain a nonzero vector subspace of $\RR^n$.
\end{Def}
\begin{Def}
      Fix a lattice $L$ in $\RR^n$.  A \emph{fan} in $\RR^n$ is a finite set $\Sigma$ of strongly convex rational polyhedral cones (with respect to $L$) such that
\begin{enumerate}[label=(\roman*),nolistsep]
    \item If $\sigma \in \Sigma$ and $\tau$ is a face of $\sigma$, then $\tau \in \Sigma$;
    \item If $\sigma, \tau \in \Sigma$, then $\sigma \cap \tau$ is a face of both.
\end{enumerate}

    A fan $\Sigma$ uniquely determines a toric variety $X_{\Sigma}$ for the torus $T = L \otimes_{\ZZ} \CC^{\times} \cong (\CC^{\times})^n$ whose lattice of one-parameter subgroups is $L$.

    We say that a fan $\Sigma$ is \emph{complete} if the union of its cones is $\RR^n$, and \emph{simplicial} if for any $\sigma \in \Sigma$, the minimal set of generators of $\sigma$ is linearly independent over $\RR$.
\end{Def}

\begin{Def}
Let $\Delta$ be an abstract simplicial complex on ground set $[1,n] = \{1, \dots, n\}$. The \emph{Stanley-Reisner ring} (or face ring) of $\Delta$ is $\CC[\Delta] = \CC[x_1, \dots, x_n]/I_\Delta$, where $I_\Delta$ is the ideal generated by the monomials $\prod_{i \in S} x_i$ for $S \notin \Delta$.
\end{Def}

\begin{Thm}\cite[Theorem 12.4.14]{CLS} \label{thm:equivCohlgy} 
    Let $\Sigma$ be a complete simplicial fan in $\R^n$.  Then the $T$-equivariant cohomology $H^*_T(X_\Sigma)$ of $X_{\Sigma}$ is isomorphic to the Stanley-Reisner ring $\CC[\Delta_\Sigma]$, where $\Delta_\Sigma$ is the abstract simplicial complex whose ground set is the set of all rays in $\Sigma$ and a set $S$ of rays in $\Sigma$ is a face of $\Delta_{\Sigma}$ if and only if $S$ consists of the extremal rays of a cone in $\Sigma$.
\end{Thm}

Next we recall some algebraic properties of the Stanley-Reisner ring.  The ground set of all simplicial complexes in this paper are assumed to be finite.  In particular, the set of faces is also finite.  An abstract simplicial complex $\Delta$ is called pure if its maximal faces have the same dimension.

\begin{Def}
    Let $\Delta$ be a pure simplicial complex.  We say that $\Delta$ is \emph{shellable} if its maximal faces can be enumerated
    \[
    F_1, F_2, \ldots, F_l
    \]
    in such a way that, for each $i \in [2,l]$, the intersection of $F_i$ with the simplicial subcomplex of $\Delta$ generated by $F_1, F_2, \ldots, F_{i-1}$ is generated by a nonempty set of maximal proper faces of $F_i$.  The enumeration $F_1, F_2, \ldots, F_l$ is called a \emph{shelling} of $\Delta$.
\end{Def}

\begin{Thm} \cite[Theorem 5.1.13]{BH98} \label{Thm:shellable}
    If $\Delta$ is a shellable simplicial complex, then the Stanley-Reisner ring $\CC[\Delta]$ is Cohen-Macaulay.
\end{Thm}

\begin{Thm} \cite[Corollary 5.6.5]{BH98} \label{Thm:simplicialSphere}
    Let $\Delta$ be a simplicial sphere, i.e., an abstract simplicial complex whose geometric realization is homeomorphic to a sphere.  Then the Stanley-Reisner ring $\CC[\Delta]$ is Gorenstein.
\end{Thm}

\section{The Case of \texorpdfstring{$\OminnnAtwo\subset\slc(\C)$}{A2}}
In the case $n=2$,
\begin{equation} \label{eq:A2}
\OminnnAtwo=\left\{\begin{pmatrix}
        0 & x_4 & x_1\\
        x_2 & 0 & x_5\\
        x_6 & x_3 & 0
    \end{pmatrix}\,\middle\vert\ \,
    \begin{tabular}{@{}c@{}c@{}} $x_2x_4=x_4x_5=x_1x_2=0$    \\ $x_2x_3=x_3x_5=x_5x_6=0$ \\
    $x_4x_6=x_1x_3=x_1x_6=0$
    \end{tabular}
\right\}.
\end{equation}
Let $\varepsilon_1,\varepsilon_2$ be the standard basis of $\R^2$.
    Let $\Phi^{C_2}=\{\pm(\varepsilon_1\pm\varepsilon_2),\pm2\varepsilon_1,\pm2\varepsilon_2\}$ be the root system of type $C_2$, with simple roots $\Pi^{C_2}=\{\alpha_1=\varepsilon_1-\varepsilon_2,\,\alpha_2=2\varepsilon_2\}$, and $\Phi^{C_2}_+=\{\alpha_1,\alpha_2,\alpha_1+\alpha_2,2\alpha_1+\alpha_2\}$ be the positive roots with respect to $\Pi^{C_2}$.
    Following \cite{FZ}, define the almost positive roots to be $$\Phi^{C_2}_{\geq-1}=\Phi^{C_2}_+\sqcup -\Pi^{C_2}.$$ 
    Then we have the following bijection between the almost positive roots $\Phi^{C_2}_{\geq-1}$ of type $C_2$ and all the roots $\Phi^{A_2}$ of type $A_2$:
\iffalse
\begin{multicols}{2}
\hspace{5em}
\begin{tikzpicture}[scale=0.8, >=stealth]
    \tikzset{label/.style={, inner sep=1pt}}
    \draw[ultra thick, brown!50!black, ->] (0,0) -- (2,0) 
        node[label, above] {$2\alpha_1+\alpha_2$};  
    \draw[ultra thick, blue, ->] (0,0) -- (0,2) 
        node[label, left] {$\alpha_2$};  
    \draw[ultra thick, brown!50!black, ->] (0,0) -- (1,1) 
        node[label, above right] {$\alpha_1+\alpha_2$};  
    \draw[ultra thick, red, ->] (0,0) -- (1,-1) 
        node[label, below right] {$\alpha_1$};  
    \draw[ultra thick, green!50!black, dashed, ->] (0,0) -- (-1,1) 
        node[label, above left] {$-\alpha_1$}; 
    \draw[ultra thick, green!50!black, dashed, ->] (0,0) -- (0,-2) 
        node[label, right] {$-\alpha_2$};  
\hspace{10em}
    $\setlength\arraycolsep{2pt}\def\arraystretch{1.5}\begin{pmatrix}
          0 & \color{red}\alpha_1 & \color{brown!50!black}2\alpha_1+\alpha_2\\
          \color{blue}\alpha_2 & 0 & \color{brown!50!black}\alpha_1+\alpha_2\\
          \color{green!50!black}-\alpha_1 & \color{green!50!black}-\alpha_2 & 0
    \end{pmatrix} $ 
    \end{tikzpicture}
\end{multicols}
\fi
\begin{multicols}{2}
\hspace{5em}
\begin{tikzpicture}[scale=0.8, >=stealth]
    \tikzset{label/.style={, inner sep=1pt}}
    \draw[ultra thick, brown!50!black, ->] (0,0) -- (2,0) 
        node[label, above] {$2\alpha_1+\alpha_2$};  
    \draw[ultra thick, blue, ->] (0,0) -- (0,2) 
        node[label, left] {$\alpha_2$};  
    \draw[ultra thick, brown!50!black, ->] (0,0) -- (1,1) 
        node[label, above right] {$\alpha_1+\alpha_2$};  
    \draw[ultra thick, red, ->] (0,0) -- (1,-1) 
        node[label, below right] {$\alpha_1$};  
    \draw[ultra thick, green!50!black, dashed, ->] (0,0) -- (-1,1) 
        node[label, above left] {$-\alpha_1$}; 
    \draw[ultra thick, green!50!black, dashed, ->] (0,0) -- (0,-2) 
        node[label, right] {$-\alpha_2$};  
\hspace{10em}
    $\setlength\arraycolsep{2pt}\def\arraystretch{1.5}\begin{pmatrix}
          0 & \color{green!50!black}-\alpha_1 & \color{green!50!black}-\alpha_2 & \\
          \color{brown!50!black}2\alpha_1+\alpha_2 & 0 & \color{red}\alpha_1 \\
          \color{brown!50!black}\alpha_1+\alpha_2 & \color{blue}\alpha_2 & 0
    \end{pmatrix} $ 
    \end{tikzpicture}
\end{multicols}

\vspace{-1em}
Notice that two almost positive roots in $\Phi^{C_2}_{\geq-1}$ are adjacent to each other if and only if they are in the same row or column under the above bijection.

Let $L \subset \RR^2$ be the lattice generated by $\Pi^{C_2}$, and $\Sigma$ be the fan in $\R^2$ whose maximal cones are spanned by adjacent almost positive roots in $\Phi^{C_2}_{\geq-1}$. Let $X_{\Sigma}$ be the toric variety, for the torus $T = L \otimes_{\ZZ} \CC^{\times} \cong (\CC^{\times})^2$ whose lattice of one-parameter subgroups is $L$, associated with the fan $\Sigma$.

\begin{Thm}
    We have
    \[
    {\C}[\OminnnAtwo]\cong H^*_T(X_{\Sigma}).
    \]
\end{Thm}

\begin{proof}
    By Theorem \ref{thm:equivCohlgy},
    \[
        H^*_T(X_{\Sigma}) \cong \CC\left[x_\alpha: \alpha\in\Phi^{C_2}_{\geq-1}\right]/\big(x_\alpha x_\beta, \textrm{ such that } \alpha,\beta \textrm{ are not adjacent}\big).
    \]
    Under the bijection above, the monomial $x_\alpha x_\beta$ for $\alpha,\beta$ non-adjacent corresponds precisely to the monomial given by a two-by-two minor of a generic element of $\slc(\C)$ restricted to $\Lien^+\oplus\Lien^-$, i.e., a monomial that appears in Equation (\ref{eq:A2}).  The desired isomorphism follows.
\end{proof}

\section{Compatibility of Almost Positive Roots of Type \texorpdfstring{$C_n$}{Cn}}
In this section, we recall the notation of compatibility degree for the root system $\Phi$ of type $C_n$.  The Dynkin diagram of the root system of type $C_n$ is
\begin{center}
    \begin{dynkinDiagram}[text style/.style={scale=1.0},
            edge length=1.2cm,
            labels={1,2,n-2,n-1,n},
            label macro/.code={\alpha_{\drlap{#1}}}
            ]C{}
    \end{dynkinDiagram}
\end{center}
Let $\varepsilon_1,\dots,\varepsilon_n$ be the standard basis of $\R^n$.  It is convenient to use the presentation
\begin{align*}
    \alpha_1 = \varepsilon_1 - \varepsilon_2, ~ \alpha_2 = \varepsilon_2 - \varepsilon_3, ~ \alpha_3 = \varepsilon_3 - \varepsilon_4, \ldots, \alpha_{n-1} = \varepsilon_{n-1} - \varepsilon_n, ~ \alpha_n = 2 \varepsilon_n
\end{align*}
of the simple roots interchangeably.  With the $\varepsilon$'s, the positive roots are
\begin{align*}
    \{\varepsilon_i \pm \varepsilon_j: 1 \le i < j \le n\} \sqcup \{2 \varepsilon_i: 1 \le i \le n\}.
\end{align*}

Fix a Coxeter word $$\underline c = (s_1, s_2, \cdots, s_{n-1}, s_n)$$ and let 
$
    c=s_1s_2\cdots s_{n-1}s_n
$
be the corresponding Coxeter element, where $s_i$ is the $i$-th simple reflection in the Weyl group of type $C_n$.  Following \cite{FZ}, for a root system $\Phi = \Phi_+ \sqcup \Phi_-$ with simple roots $\Pi$, define the set of almost positive roots to be 
\[
\Phi_{\geq -1} = \Phi_+ \sqcup - \Pi.
\]
\begin{Def}
    Following \cite{CP}, define a piecewise linear map $\tau:\Phi_{\ge -1} \rightarrow \Phi_{\ge -1}$
\begin{align*}
    \tau(\alpha)=
    \begin{cases}
        s_1 s_2 \cdots s_{i-1} (\alpha_i) ~ & \text{if} ~ \alpha = -\alpha_i ~ \exists i \in [1,n] \\
        -\alpha_i ~ & \text{if} ~ \alpha = s_n s_{n-1} \cdots s_{i+1} (\alpha_i) ~ \exists i \in [1,n] \\
        c(\alpha) ~ & \text{otherwise}.
    \end{cases}
\end{align*}
\end{Def}

If we would like to emphasize that a root system $\Phi$ is of Dynkin type $X$, we often indicate this by writing $\Phi^X$ in place of $\Phi$.  We define a bijection $\Phi^{C_n}_{\ge -1}\stackrel{1:1}{\longrightarrow}\Phi^{A_n}$ by placing each element $\alpha\in\Phi^{C_n}_{\ge -1}$ in the following array such that the location of $\alpha$ corresponds to the root space of the image of $\alpha$ in $\Phi^{A_n}$ under the bijection.
\begin{align} \label{Eqn:arrTypeA}
\scalebox{1}{
$
\NiceMatrixOptions{margin=4pt,renew-matrix}
\setlength\arraycolsep{0.33em}\def\arraystretch{1.33}
\begin{pNiceArray}{cccccccc}
    \CodeBefore
%    \tikz \draw [fill=red!15] (1-|2) |- (1-|8) |- (7-|8) |- (7-|7) |- (6-|7) |- (6-|6) |- (5-|6) |- (5-|5) |- (4-|5) |- (4-|4) |- (3-|4) |- (3-|3) |- (2-|3) |- (2-|2) |- cycle;
%    \tikz \draw [fill=blue!15] (1-|8) |- (1-|9) |- (8-|9) |- (8-|8) |- cycle;
%    \tikz \draw [fill=green!15] (8-|1) |- (8-|8) |- (9-|8) |- (9-|1) |- cycle;
%    \tikz \draw [fill=yellow!15] (2-|1) |- (2-|2) |- (3-|2) |- (3-|3) |- (4-|3) |- (4-|4) |- (5-|4) |- (5-|5) |- (6-|5) |- (6-|6) |- (7-|6) |- (7-|7) |- (8-|7) |- (8-|1) |- cycle;
    \Body
    0 & -\varepsilon_1\!+\!\varepsilon_2 & -\varepsilon_2 \!+\!\varepsilon_3 & -\varepsilon_3\!+\!\varepsilon_4 & -\varepsilon_4\!+\!\varepsilon_5 & \cdots & -2 \varepsilon_n \\
    2\varepsilon_1 & 0 & \varepsilon_1\!-\!\varepsilon_2 & \varepsilon_1\!-\!\varepsilon_3 & \varepsilon_1\!-\!\varepsilon_4 & \cdots & \varepsilon_1\!-\!\varepsilon_n \\
    \varepsilon_1\!+\!\varepsilon_2  & 2\varepsilon_2 & 0 & \varepsilon_2\!-\!\varepsilon_3 & \varepsilon_2\!-\!\varepsilon_4 & \cdots & \varepsilon_2\!-\!\varepsilon_n \\
    \varepsilon_1\!+\!\varepsilon_3  & \varepsilon_2\!+\!\varepsilon_3 & 2\varepsilon_3 & 0 & \varepsilon_3\!-\!\varepsilon_4 & \cdots & \varepsilon_3\!-\!\varepsilon_n \\
    \vdots & \vdots & \vdots & \ddots & \ddots & \ddots & \vdots \\
    \varepsilon_1\!+\!\varepsilon_{n-1} & \varepsilon_2\!+\!\varepsilon_{n-1} & \varepsilon_3\!+\!\varepsilon_{n-1}  & \cdots & 2\varepsilon_{n-1} & 0 & \varepsilon_{n-1}\!-\!\varepsilon_n \\
    \varepsilon_1\!+\!\varepsilon_n & \varepsilon_2\!+\!\varepsilon_n & \varepsilon_3\!+\!\varepsilon_n  & \cdots & \varepsilon_{n-1}\!+\!\varepsilon_n & 2\varepsilon_n & 0  
\end{pNiceArray}.
$
}
\tag{$\dagger$}
\end{align}

\begin{Lem} \label{Cor:tauOrb}
    The $\tau$-orbits in $\Phi_{\ge -1}$ are given by
    \begin{align*}
        & -\varepsilon_1+\varepsilon_2 \taumapsto  \varepsilon_1-\varepsilon_2 \taumapsto \varepsilon_2-\varepsilon_3 \taumapsto \varepsilon_3-\varepsilon_4 \taumapsto \cdots \taumapsto \varepsilon_{n-1} - \varepsilon_n \taumapsto \varepsilon_n+\varepsilon_1 , \\
        &-\varepsilon_2+\varepsilon_3 \taumapsto  \varepsilon_1-\varepsilon_3 \taumapsto \varepsilon_2-\varepsilon_4 \taumapsto \varepsilon_3-\varepsilon_5 \taumapsto \cdots \taumapsto \varepsilon_{n-1}+\varepsilon_1 \taumapsto \varepsilon_n+\varepsilon_2 , \\
        & \hspace{15.5em} \vdots \\
        &-\varepsilon_{n-1}+\varepsilon_n \taumapsto  \varepsilon_1-\varepsilon_n \taumapsto \varepsilon_2+\varepsilon_1 \taumapsto \varepsilon_3+\varepsilon_2 \taumapsto \cdots \taumapsto \varepsilon_{n-1}+\varepsilon_{n-2} \taumapsto \varepsilon_n+\varepsilon_{n-1} , \\
        & -2\varepsilon_n \taumapsto  2\varepsilon_1 \taumapsto 2\varepsilon_2 \taumapsto 2\varepsilon_3 \taumapsto \cdots \taumapsto 2\varepsilon_{n-1} \taumapsto 2\varepsilon_n.
    \end{align*}
\end{Lem}
\begin{proof}
We first compute
    \[
        c(\varepsilon_1) = \varepsilon_2, ~ c(\varepsilon_2) = \varepsilon_3, ~ \ldots, c(\varepsilon_{n-1}) = \varepsilon_n, ~ c(\varepsilon_n) = - \varepsilon_1.
    \]
Repeatedly applying the Coxeter element $c$, we get
    \begin{align*}
        & \varepsilon_1-\varepsilon_2 \cmapsto \varepsilon_2-\varepsilon_3 \cmapsto \varepsilon_3-\varepsilon_4 \cmapsto \cdots \cmapsto \varepsilon_{n-2}-\varepsilon_{n-1} \cmapsto \varepsilon_{n-1}-\varepsilon_n \cmapsto \varepsilon_n+\varepsilon_1, \\
        & \varepsilon_1-\varepsilon_3 \cmapsto \varepsilon_2-\varepsilon_4 \cmapsto \varepsilon_3-\varepsilon_5 \cmapsto \cdots \cmapsto \varepsilon_{n-2}-\varepsilon_n \cmapsto \varepsilon_{n-1}+\varepsilon_1 \cmapsto \varepsilon_n+\varepsilon_2, \\
        & \hspace{15.5em}\vdots \\
        & \varepsilon_1-\varepsilon_n \cmapsto \varepsilon_2+\varepsilon_1 \cmapsto \varepsilon_3+\varepsilon_2 \cmapsto \cdots \cmapsto  \varepsilon_{n-2}+\varepsilon_{n-3} \cmapsto \varepsilon_{n-1}+\varepsilon_{n-2} \cmapsto \varepsilon_n+\varepsilon_{n-1}, \\
        & 2\varepsilon_1 \cmapsto 2\varepsilon_2 \cmapsto 2\varepsilon_3 \cmapsto \cdots \cmapsto 2\varepsilon_{n-2} \cmapsto 2\varepsilon_{n-1} \cmapsto 2\varepsilon_n. 
    \end{align*}
    Observe that
\begin{align*}
    \{s_n s_{n-1} \cdots s_{i+1} (\alpha_i): i \in [1,n]\} = \{2\varepsilon_n, \varepsilon_{n-1}+\varepsilon_n, \varepsilon_{n-2}+\varepsilon_n, \ldots, \varepsilon_1+\varepsilon_n\},
\end{align*}
So
\begin{align*}
    \tau (\varepsilon_i+\varepsilon_n) =
    \begin{cases}
        \varepsilon_{i+1}-\varepsilon_i ~ & \text{if} ~ i \in [1,n-1] \\
        -2\varepsilon_n ~ & \text{if} ~ i = n.
    \end{cases}
\end{align*}
The statement follows.
\qedhere
\end{proof}

\begin{rem}
In terms of the bijection (\ref{Eqn:arrTypeA}), the $\tau$-orbits in $\Phi_{\ge -1}$ are precisely the ``slope $(-1)$ curves on the torus''.  More specifically, $\tau$ sends the almost positive root in the $(i,j)$-entry of (\ref{Eqn:arrTypeA}) to the $(k,l)$-entry of (\ref{Eqn:arrTypeA}) if and only if $$k \equiv i+1 ~ (\text{mod} ~ n+1) \quad \text{and} \quad l \equiv j+1 ~ (\text{mod} ~ n+1).$$
\end{rem}

\begin{Def} Following \cite{CP}, let
$
    (-|\!|_{\underline c}-): \Phi_{\ge -1} \times \Phi_{\ge -1} \rightarrow \ZZ
$
be the unique $\tau$-invariant function such that 
\begin{enumerate}[label=(\roman*),nolistsep]
    \item For any $\alpha, \alpha' \in \Pi$, we have $(-\alpha |\!|_{\underline c} -\alpha') = 0$;
    \item For any $\alpha \in \Pi$ and $\beta \in \Phi_+$,
    \[
    (-\alpha |\!|_{\underline c} \beta) = [\beta:\alpha], \quad \text{where } \beta = \sum_{\alpha'' \in \Pi} [\beta:\alpha''] \alpha''.
    \]
\end{enumerate}
For $\gamma,\gamma'\in\Phi_{\geq-1}$, the number $(\gamma |\!|_{\underline c} \gamma')$ is called the \emph{$\underline{c}$-compatibility degree} of $\gamma$ and $\gamma'$.  We say that $\gamma$ and $\gamma'$ are \emph{${\underline c}$-compatible} if $(\gamma_1|\!|_{\underline c}\gamma_2)=0$.
\end{Def}
\begin{Ex} For the root system $\Phi^{C_2}$ of type $C_2$ the $\underline{c}$-compatibility degrees are given in the following table:
    \begin{center}
\begin{tabular}{|c|c|c|c|c|c|c|}
\hline
        & $-\alpha_1$ & $-\alpha_2$ & $\alpha_1$& $\alpha_2$ & $\alpha_1+\alpha_2$ & $2\alpha_1+\alpha_2$ \\ \hline

$-\alpha_1$          & $0$ & $0$ & $1$ & $0$ & $1$ & $2$\\ \hline
$-\alpha_2$          & $0$ & $0$ & $0$ & $1$ & $1$ & $1$\\ \hline        
$\alpha_1$           & $1$ & $0$ & $0$ & $2$ & $1$ & $0$\\ \hline
$\alpha_2$           & $0$ & $1$ & $1$ & $0$ & $0$ & $1$\\ \hline
$\alpha_1+\alpha_2$  & $1$ & $2$ & $1$ & $0$ & $0$ & $0$\\ \hline
$2\alpha_1+\alpha_2$ & $1$ & $1$ & $0$ & $1$ & $0$ & $0$\\ \hline
\end{tabular}
\end{center}
So indeed ``$\underline{c}$-compatibility'' generalizes the notion of ``adjacency''.
\end{Ex}

For integers $1\leq k<m \leq n+1$ and $ 1\leq i<j \leq n+1$, we use $(k,m\mid i,j)$ to denote the two-by-two submatrix
$$
    \setlength\arraycolsep{1.3pt}\begin{pmatrix}
        x_{ki} && x_{kj}\\
        x_{mi} && x_{mj} 
    \end{pmatrix}
$$
of a generic $(n+1)$-by-$(n+1)$ matrix.  The main theorem of this section is the following result.
\begin{Thm}\label{compatible}
Two almost positive roots $\alpha,\beta\in\Phi_{\geq-1}$ are $\underline c$-compatible if and only if, under the bijection (\ref{Eqn:arrTypeA}), the locations of $\alpha$ and $\beta$ satisfy one of the following conditions:
\begin{enumerate}[label=(\roman*),nolistsep]
    \item They are in the same row or the same column.
    \item They are in the $(k,i)$- and $(m,j)$-entries of a two-by-two submatrix $(k,m\mid i,j)$ disjoint from the diagonal, with ($k<i<m$) or ($i<k<m<j$).
    \item They are in the $(m,i)$- and $(k,j)$-entries of a two-by-two submatrix $(k,m\mid i,j)$ disjoint from the diagonal, with ($k<i$ and $m<i$) or ($i<k<j<m$) or ($k>j$).
\end{enumerate}
\end{Thm}
For simplicity, from now on, the term ``locations of almost positive roots'' should be understood under the bijection ($\dagger$).  First we prove several lemmas.
\begin{Lem}\label{samerowcolumn}
    Every pair of almost positive roots in the same row or column of the matrix (\ref{Eqn:arrTypeA}) are $\underline c$-compatible.
\end{Lem}

\begin{proof}
    Suppose that $\alpha, \beta \in \Phi_{\ge -1}$ are in the same row of (\ref{Eqn:arrTypeA}).  By the description of the map $\tau$ in Lemma \ref{Cor:tauOrb}, we see that $\tau(\alpha), \tau(\beta)$ are also in the same row of (\ref{Eqn:arrTypeA}).  Hence, by $\tau$-invariance of the $\underline c$-compatibility degree, we may assume without loss of generality that $\alpha$ and $\beta$ are in the first row of (\ref{Eqn:arrTypeA}).  Since all roots in the first row of (\ref{Eqn:arrTypeA}) are negative simple roots, $\alpha$ and $\beta$ are $\underline c$-compatible.

    Suppose that $\alpha, \beta \in \Phi_{\ge -1}$ are in the same column of (\ref{Eqn:arrTypeA}).  By a similar argument to the previous paragraph, we may assume without loss of generality that one of $\alpha, \beta$ is in the first row of (\ref{Eqn:arrTypeA}).  Hence, it suffices to show that, for all $i \in [1,n]$, $$[\gamma:\alpha_i] = 0$$ for each $\gamma$ which is in the same column of (\ref{Eqn:arrTypeA}) as $-\alpha_i$.  But this is easily done by inspection.
\end{proof}

\begin{Lem}\label{onezeroentry}
    For any distinct $i,j,k \in [1,n+1]$, the almost positive roots in the $(i,j)$- and $(j,k)$-entries of (\ref{Eqn:arrTypeA}) are not $\underline c$-compatible.
\end{Lem}

\begin{proof}
    Again using the description of the map $\tau$ in Lemma \ref{Cor:tauOrb}, we may assume without loss of generality that $j = 1$.  Since for each almost positive root $\beta$ in the first column of (\ref{Eqn:arrTypeA}), we have
    \[
    [\beta : \alpha] \neq 0 \quad \forall \alpha \in \Pi,
    \]
    and the first row of (\ref{Eqn:arrTypeA}) consists of negative simple roots, an almost positive root in the first column of (\ref{Eqn:arrTypeA}) can never be $\underline c$-compatible with an almost positive root in the first row of (\ref{Eqn:arrTypeA}).
\end{proof}

\begin{Lem}\label{onlyonediag}
    Let $i,j,k,l \in [1,n+1]$ be distinct.  Assume that $i < k$ and $j < l$.  Then
    \begin{align*}
        & \text{The} ~ (i,j) \text{-entry of (\ref{Eqn:arrTypeA}) is} ~ \underline c ~ \text{-compatible with the} ~ (k,l) \text{-entry of (\ref{Eqn:arrTypeA})} \\
        \iff & \text{The} ~ (i,l) \text{-entry of (\ref{Eqn:arrTypeA}) is not} ~ \underline c ~ \text{-compatible with the} ~ (k,j) \text{-entry of (\ref{Eqn:arrTypeA})}.
    \end{align*}
\end{Lem}

\begin{proof}
    As above, we may assume without loss of generality that $k = 1$.

    For $p \in [1,n]$, those almost positive roots that are not $\underline c$-compatible with $- \alpha_p$ are depicted below:
    \begin{align*}
    \NiceMatrixOptions{margin=4pt,renew-matrix}
        \begin{pNiceArray}{ccccccccc}
            \CodeBefore
            \tikz \draw [fill=red!40!black] (2-|7) |- (2-|10) |- (7-|10) |- (7-|7) |- cycle;
            % \tikz \draw [fill=brown!15] (1-|10) |- (1-|11) |- (10-|11) |- (10-|10) |- cycle;
            \tikz \draw [fill=red!40!black] (2-|1) |- (2-|2) |- (3-|2) |- (3-|3) |- (4-|3) |- (4-|4) |- (5-|5) |- (6-|5) |- (6-|6) |- (10-|6) |- (10-|1) |- cycle;
            \Body
            0 &&&&& -\alpha_p &&& \\
            & 0 &&&&&&& \\
            && 0 &&&&&& \\
            &&& 0 &&&&& \\
            &&&& \ddots &&&& \\
            &&&&& 0 & \textcolor{white}{\alpha_p} & \\
            &&&&&& 0 && \\
            &&&&&&& 0 & \\
            &&&&&&&& \ddots
        \end{pNiceArray}\tag{$*$}
    \end{align*}
    The assertion is clear from this picture.
\end{proof}

\begin{proof}[Proof of Theorem \ref{compatible}]
Let $\alpha, \beta \in \Phi_{\ge -1}$. If $\alpha,\beta$ are in the same row or column, then by Lemma \ref{samerowcolumn} they are $\underline{c}$-compatible. If $\alpha,\beta$ are in the $(m,i)$- and $(i,j)$-entries for some $m,i,j$, then by Lemma \ref{onezeroentry}, $\alpha$ and $\beta$ are not $\underline{c}$-compatible.

From the discussion in the previous paragraph, we may now assume that the locations of $\alpha, \beta$ are in a two-by-two submatrix $(k,m\mid i,j)$, where $k,m,i,j$ are distinct (and $k<m$ and $i<j$).
%From the discussion in the previous paragraph, we may assume $k\neq i, k\neq j, m\neq i, m\neq j$.
First observe that if $k=1$, then we reduce the question of $\underline{c}$-compatibility of $\alpha$ and $\beta$ back to $(*)$ above, i.e. if $m<i$, then the almost positive roots in the $(m,i)$- and $(k,j)$-entries are $\underline{c}$-compatible, and the almost positive roots in the $(m,j)$- and $(k,i)$-entries are not; otherwise those in the $(k,i)$- and $(m,j)$-entries are $\underline{c}$-compatible, and those in the $(k,j)$- and $(m,i)$-entries are not.

For general $k$, we apply $\tau^{-1}$ to the $(n+1)$-by-$(n+1)$ matrix $(k-1)$ times, then $(k,m\mid i,j)$ becomes $(k',m'\mid i',j')$ where $k'=1,\,m'=m-k+1$ and $i',j'$ are computed case by case as follows.
\begin{enumerate}[label=(\roman*)]
    \item If $k<i$, then $i'=i-k+1$ and $j'=j-k+1$. %(i.e. ``No Flip")
    \item If $i<k<j$, then $i'=j-k+1$ and $j'=(i+n+1)-k+1$. %(i.e. ``Flip Once")
    \item If $k>j$, then $i'=(i+n+1)-k+1$ and $j'=(j+n+1)-k+1$. %(i.e. ``Flip Twice")
\end{enumerate}
Since $(-|\!|_{\underline{c}}-)$ is $\tau$-invariant, we have reduced to the case of $k=1$ above. Explicitly we have:
\begin{enumerate}[label=(\roman*)]
    \item If $k<i$ and $m<i$, then the $(m,i)$- and $(k,j)$-entries are $\underline{c}$-compatible, and the $(m,j)$- and $(k,i)$-entries are not $\underline{c}$-compatible.
    \item If $k<i$ and $m>i$, then the $(k,i)$- and $(m,j)$-entries are $\underline{c}$-compatible, and the $(k,j)$- and $(m,i)$-entries are not $\underline{c}$-compatible.
    \item If $i<k<j$ and $m<j$, then the $(k,i)$- and $(m,j)$-entries are $\underline{c}$-compatible, and the $(k,j)$- and $(m,i)$-entries are not $\underline{c}$-compatible.
    \item If $i<k<j$ and $m>j$, then the $(m,i)$- and $(k,j)$-entries are $\underline{c}$-compatible, and the $(m,j)$- and $(k,i)$-entries are not $\underline{c}$-compatible.
    \item If $k>j$, since $m<i+n+1$ always holds, then the $(m,i)$- and $(k,j)$-entries are $\underline{c}$-compatible, and the $(m,j)$- and $(k,i)$-entries are not $\underline{c}$-compatible.
\end{enumerate}
And our statement follows.
\qedhere
\end{proof}

\section{A Gr\"obner Basis Result for \texorpdfstring{$m$}{m}-by-\texorpdfstring{$n$}{n} Matrices with Prescribed Zero Entries}
To study the coordinate ring $\CC[\Ominnn]$, we employ techniques from computational algebraic geometry.  We refer to \cite{CLO} for more details of the definitions and results cited below. Let $R = \CC[x_1, \dots, x_N]$ be a polynomial ring.

\begin{Def}
A \emph{monomial order} on $R$ is a total order $>$ on the set of monomials such that:
\begin{enumerate}[label=(\roman*)]
\item $m \geq 1$ for all monomials $m$;
\item If $m_1 > m_2$, then $mm_1 > mm_2$ for any monomial $m$.
\end{enumerate}
\end{Def}

An important example of a monomial order is the
\emph{degree reverse lexicographic order} (degree revlex order): $$x_1^{a_1}\cdots x_N^{a_N} > x_1^{b_1}\cdots x_N^{b_N}$$ if $\sum a_i > \sum b_i$, or $\sum a_i = \sum b_i$ and the rightmost nonzero entry of $(a_1-b_1, \dots, a_N-b_N)$ is negative.

\begin{Def}
For a nonzero polynomial $f = \sum c_\gamma x^\gamma \in R$ and a monomial order $>$, the \emph{initial monomial} $\init_>(f)$ is the largest monomial $x^\gamma$ with $c_\gamma \neq 0$. For an ideal $I \subset R$, the \emph{initial ideal} is $\init_>(I) = \big( \init_>(f) : f \in I \big)$.
\end{Def}

\begin{Def}
A finite generating subset $G = \{g_1, \dots, g_k\}$ of an ideal $I \subset R$ is a \emph{Gr\"obner basis} for $I$ with respect to $>$ if $\init_>(I) = \big( \init_>(g_1), \dots, \init_>(g_k) \big)$.
\end{Def}

A key property of Gr\"obner bases and initial ideals that we will use is the following well-known theorem of Macaulay.

\begin{Thm}[Macaulay] \label{Thm:Macaulay}
If $I$ is a homogeneous ideal of $R$ and $>$ is a monomial order, then the Hilbert series of the homogeneous rings $R/I$ and $R/\init_>(I)$ coincide.
\end{Thm}

For any positive integers $m,n$, let ${\rm Mat}_{m \times n}$ be the vector space of $m$ by $n$ matrices with complex entries.  Let $Z \subset [1,m] \times [1,n]$.  Define a closed subvariety
\[
{\rm Mat}_{m \times n}(Z) = \{ X = (x_{ij}) \in {\rm Mat}_{m \times n} \colon x_{ij} = 0 ~ \forall (i,j) \in Z \}
\]
of ${\rm Mat}_{m \times n}$, so the coordinate ring $\CC[{\rm Mat}_{m \times n}(Z)]$ of ${\rm Mat}_{m \times n}(Z)$ can be naturally identified with the polynomial ring $\CC[x_{ij} \colon (i,j) \in [1,m] \times [1,n] \setminus Z]$.  Let
\begin{align*}
    & \mathcal{G}_2(Z) = \{\text{two-by-two minors of a generic } m \times n \text{ matrix that belongs to } {\rm Mat}_{m \times n}(Z)\}, \\
    & I_2(Z) = \text{ideal of } \CC[{\rm Mat}_{m \times n}(Z)] \text{ generated by } \mathcal{G}_2(Z).
\end{align*}

\begin{Thm} \label{Thm:2by2Groebner}
    The set $\mathcal{G}_2(Z)$ is a universal degree revlex Gr\"obner basis of $I_2(Z)$, i.e., for any total order of the indeterminates of $\CC[{\rm Mat}_{m \times n}(Z)]$, $\mathcal{G}_2(Z)$ is a Gr\"obner basis of $I_2(Z)$ with respect to the degree revlex order induced by that total order.
\end{Thm}

\begin{proof}
    Let $f, g \in \mathcal{G}_2(Z)$.  By the Buchberger criterion, we must show that the $S$-polynomial $S(f,g)$ of $f$ and $g$ reduces to $0$ with respect to $\mathcal{G}_2(Z)$.  A two-by-two minor of a generic matrix that belongs to ${\rm Mat}_{m \times n}(Z)$ is either a monomial or a binomial.  Hence, for $f$ and $g$ we have the following possibilities: (i) $f$ and $g$ are monomials; (ii) $f$ is a monomial and $g$ is a binomial; (iii) $f$ and $g$ are binomials.

    In case (i), by definition we have $S(f,g) = 0$, and there is nothing to prove.

    We recall that if the initial monomials of $f$ and $g$ are coprime, then $S(f,g)$ automatically reduces to $0$ (with respect to $\{f,g\}$).  Hence, for the rest of this proof, it suffices to consider the situation where the initial monomials of $f$ and $g$ are not coprime.  Consequently, in case (ii), we may assume without loss of generality that
    \[
    f = x_{kl} x_{pq}, \quad g = x_{ij} x_{kl} - x_{il} x_{kj},
    \]
    and the initial monomial of $g$ is $x_{ij} x_{kl}$.  It is evident that $S(f,g) = x_{il} x_{kj} x_{pq}$.  Note that $f = x_{kl} x_{pq}$ implies that at least one of $(k,q)$ and $(p,l)$ belongs to $Z$.  If $(k,q) \in Z$, then $\pm x_{kj} x_{pq} \in \mathcal{G}_2(Z)$; if $(p,l) \in Z$, then $\pm x_{il} x_{pq} \in \mathcal{G}_2(Z)$.  In both cases $S(f,g)$ reduces to $0$ with respect to $\mathcal{G}_2(Z)$.

    In case (iii), we may assume without loss of generality that $f$ and $g$ are two-by-two minors of a $3 \times 3$ submatrix $M$ of a generic element of ${\rm Mat}_{m \times n}(Z)$.  For the sake of ease of presentation, we permute the rows and columns of $M$ so that $f$ (resp. $g$) is the top-left (resp. bottom-right) two-by-two minor of $M$.  If all entries of $M$ are nonzero for a generic element of ${\rm Mat}_{m \times n}(Z)$, then $S(f,g)$ reduces to $0$ (with respect to the two-by-two minors of $M$, hence, with respect to $\mathcal{G}_2(Z)$), since the set of two-by-two minors of a matrix is a universal degree revlex Gr\"obner basis of the ideal it generates (see, for example, \cite[Proposition 5.3.8]{BCRV}).  If some entry of $M$ is zero for all elements of ${\rm Mat}_{m \times n}(Z)$, then it can only be the top-right or bottom-left entry.  We assume without loss of generality that the top-right entry of $M$ is zero.  Then
    \[
    \pm m_{11} m_{23}, \pm m_{12} m_{33} \in \mathcal{G}_2(Z).
    \]
    Since
    \begin{align*}
        S(f,g) & = S(m_{11} m_{22} - m_{12} m_{21}, m_{22} m_{33} - m_{23} m_{32}) \\
        & = m_{11} m_{23} m_{32} - m_{12} m_{21} m_{33} = (m_{11} m_{23}) m_{32} - (m_{12} m_{33}) m_{21},
    \end{align*}
    we see that $S(f,g)$ reduces to $0$ with respect to $\mathcal{G}_2(Z)$ in this case as well.
\end{proof}

\section{The Affine Scheme \texorpdfstring{$\Ominnn$}{Ominnn}}

Fix a presentation $\CC[\mathfrak{n}^+\oplus\mathfrak{n}^-]=\CC[x_1,x_2,\ldots,x_{n^2+n}]$, where we label the $(i,j)$-th coordinate (with $i\neq j$) by 
$$
    %a_{ij}=
    \begin{cases}
        x_{(n+1)(i-j+n)+i}, & \text{if } i<j, \\
        x_{(n+1)(i-j-1)+i}, & \text{if } i>j,
    \end{cases}
$$
that is
\begin{equation*}
\setlength\arraycolsep{0.5em}\def\arraystretch{1.5}
\begin{pmatrix}
        0 & x_{n^2} & x_{n^2-n-1} &  \cdots &  x_{2n+3} & x_{n+2} & x_{1}\\
        x_{2} & 0 & x_{n^2+1} & x_{n^2-n} & \cdots & x_{2n+4} & x_{n+3}\\
        x_{n+4} & x_{3} & 0 & x_{n^2+2} &  x_{n^2-n+1}  & \cdots & x_{2n+5} \\
        x_{2n+6} & x_{n+5} & x_{4} & \ddots  & \ddots  & \ddots  & \vdots \\
        \vdots & \ddots & \ddots & \ddots & 0 & x_{n^2+n-2}  & x_{n^2-3}\\
        x_{n^2-2} &  \cdots & x_{3n+2} & x_{2n+1} & x_{n}  &0  & x_{n^2+n-1} \\
        x_{n^2+n} & x_{n^2-1}  & \cdots & x_{3n+3} &  x_{2n+2} & x_{n+1} & 0
    \end{pmatrix} \tag{$\spadesuit$}.
\end{equation*}

%%%%%%%
\iffalse
\begin{Cor} \label{cor:compatible}
    The leading monomial of a two-by-two minor of ($\spadesuit$) with respect to the degree revlex order on $\CC[x_1, x_2, \ldots, x_{n^2+n}]$ is
\end{Cor}

\begin{proof}
    This follows from Theorem \ref{compatible} and the definition of the degree revlex order.
\end{proof}
\fi
%%%%%%%

Let $Z = \{(1,1), (2,2), \ldots, (n+1,n+1)\} \subset [1,n+1] \times [1,n+1]$, so ${\rm Mat}_{(n+1) \times (n+1)}(Z) \cong \Lien^+ \oplus \Lien^-$ and the ideal $I_2(Z)$ defines the affine scheme $\Ominbar \cap (\Lien^+ \oplus \Lien^-)$.
\begin{Cor}
    A universal degree revlex Gr\"obner basis of the ideal $I_2(Z)$ of the polynomial ring $\CC[x_1, x_2, \ldots, x_{n^2+n}]$ is given by
    \[
    \mathcal{G}_2(Z) = \{ \text{two-by-two minors of the matrix (} \spadesuit \text{)} \}.
    \]
\end{Cor}
\begin{proof}
    Apply Theorem \ref{Thm:2by2Groebner} to ${\rm Mat}_{(n+1) \times (n+1)}$ and $Z = \{(1,1), (2,2), \ldots, (n+1,n+1)\}$.
\end{proof}

Let $L \subset \RR^n$ be the lattice spanned by the roots $\Phi$ of type $C_n$.  Let $\Sigma$ be the fan in $\RR^n$ whose strongly convex polyhedral cones are generated by $\underline{c}$-compatible subsets of almost positive roots $\Phi_{\geq-1}$ of type $C_n$. Write $X_\Sigma$ for the projective toric variety (for the torus $T = L \otimes_{\ZZ} \CC^{\times} \cong  (\CC^{\times})^n$) associated with the fan $\Sigma$.  Let $\CC^{\times}$ act on $\slnn(\CC)$ be scaling the matrix entries.  It is well-known that the nilpotent adjoint orbits are stable under this action.  Hence, we get a $\CC^{\times}$-action on $\Ominnn$.  Equivalently, the coordinate ring $\CC[\Ominnn]$ becomes a $\ZZ$-graded algebra whose degree $k$ component is
\[
\{f \in \CC[\Ominnn] \colon z \cdot f = z^kf, ~ \forall z \in \CC^{\times}\}.
\]
\begin{Thm} \label{Thm:degeneration} The coordinate ring $\CC[\Ominnn]$ degenerates flatly to the $T$-equivariant cohomology $H^*_T(X_{\Sigma})$, i.e., %there exists a flat family over the curve $\mathbb C$ whose fiber over $z \in \CC^{\times} \subset \CC$ is $\Ominnn$ and whose fiber over $0 \in \CC$ is $\Spec H^*_T(X_{\Sigma})$. 
there exists a flat morphism $\pi: \mathfrak{X} \rightarrow \CC$ such that
\begin{center}
    \begin{tikzcd}
	    {{\rm Spec} H^*_T(X_{\Sigma}) \cong \pi^{-1}(0)} & {\mathfrak{X}} & {\pi^{-1}(\CC^{\times}) \cong \left( \Ominnn \right) \times \CC^{\times}} \\
	    0 & {\CC} & {\CC^{\times}}
	    \arrow[hook, from=1-1, to=1-2]
	    \arrow[from=1-1, to=2-1]
	    \arrow["\pi", from=1-2, to=2-2]
	    \arrow[hook', from=1-3, to=1-2]
	    \arrow["{{\rm pr}_2}", from=1-3, to=2-3]
	    \arrow[hook, from=2-1, to=2-2]
	    \arrow[hook', from=2-3, to=2-2]
    \end{tikzcd}
\end{center}
Moreover, the Hilbert series of $\C[\Ominnn]$ is given by
    \[
    h_{\C[\Ominnn]} (t) = \frac{\sum \limits_{k=0}^n \binom{n}{k}^2 t^k}{(1-t)^n}.
    \]
\end{Thm}
\begin{proof}
    Recall the abstract simplicial complex $\Delta_{\Sigma}$ in Theorem \ref{thm:equivCohlgy}, as well as the isomorphism  %Put $S_\Sigma=\{\textrm{rays in }\Sigma\}$.  %, and define a simplicial complex $\Delta_\Sigma$ on $S_\Sigma$ by declaring that a subset of $S$ is a simplex if and only if the cone it spans is a cone in $\Sigma$. Then we have an isomorphism between the $T$-equivariant cohomology $H^*_T(X_{\Sigma})$ and the Stanley-Reisner ring
    %By Theorem \ref{thm:equivCohlgy},
    $$
        H^*_T(X_{\Sigma}) \cong \CC[\Delta_{\Sigma}].%\coloneqq\CC\left[x_s: s\in S_\Sigma\right]/\big(\prod_{s\in T}s_s: T\notin \Delta_\Sigma\big).
    $$
    By Theorem \ref{compatible} and the definition of the degree revlex order, we observe that the initial monomials of the Gr\"obner basis $\mathcal G_2(Z)$ correspond precisely to pairs of almost positive roots which are not $\underline{c}$-compatible. Now, our first statement follows from \cite[Theorem 4.4.10]{LB}.
    %A well-known theorem of Macaulay states that the Hilbert series of a homogeneous ideal agrees with that of its initial ideal with respect to any monomial order. 
    By Theorem \ref{Thm:Macaulay}, $h_{\Ominnn} (t) = h_{\CC[\Delta_{\Sigma}]} (t)$. The Hilbert series of $\CC[\Delta_{\Sigma}]$ is equal to the quotient of its $h$-polynomial by $(1-t)^n$ (see, for example, \cite[Theorem 6.15]{FMS}).  But the $h$-vector of $\Delta_{\Sigma}$ is $(\binom{n}{0}^2, \binom{n}{1}^2, \ldots, \binom{n}{n}^2)$, also known as Narayana numbers of type $C_n$ (see, for example, \cite{FR}).
\end{proof}

\begin{Thm} \label{Cor:degeneration}
    The scheme-theoretic intersection $\Ominnn$ is reduced, and Gorenstein (so, in particular, Cohen-Macaulay).
\end{Thm}

\begin{proof}
    Let $R$ be a polynomial ring and $>$ be a monomial order on $R$.  It is well-known that (see, for example, \cite[Proposition 1.6.2]{BCRV}), for any ideal $I$ of $R$, if $R/\init_>(I)$ is reduced (resp. Gorenstein), so is $R/I$.  Hence, by Theorem \ref{Thm:degeneration}, it suffices to prove that the Stanley-Reisner ring $\CC[\Delta_{\Sigma}]$ is reduced and Gorenstein.

    It is well-known that Stanley-Reisner rings are reduced.  By \cite[Theorem 1.4]{CFZ02}, the simplicial complex $\Delta_{\Sigma}$ is a simplicial sphere.  Hence, by Theorem \ref{Thm:simplicialSphere}, $\CC[\Delta_{\Sigma}]$ is Gorenstein.
\end{proof}

Another consequence of Theorem \ref{Thm:degeneration} is an elementary proof of the property that $\Ominbar$ is Gorenstein.  Our proof, in particular, does not invoke the theory of symplectic singularities.

\begin{Cor}\label{Cor:OminGor}
    The scheme $\Ominbar$ is Gorenstein.
\end{Cor}

\begin{proof}
    Let $R$ be a $\ZZ$-graded Noetherian $\CC$-algebra and $r \in R$ an element of degree $d$.  We have an exact sequence of $\ZZ$-graded Noetherian $\CC$-algebras
    \[
    0 \longrightarrow K \longrightarrow R[-d] \xrightarrow{~~\cdot r~~} R \longrightarrow R/(r) \longrightarrow 0,
    \]
    where $K$ stands for the kernel of the map $R[-d] \rightarrow R$ given by multiplication by $r$.  Hence, for the Hilbert series, we have
    \[
    h_K(t) - h_{R[-d]}(t) + h_R(t) - h_{R/(r)}(t) = 0.
    \]
    It follows that $r$ is a non-zero-divisor if and only if $- h_{R[-d]}(t) + h_R(t) - h_{R/(r)}(t) = 0$, i.e.,
    \[
    h_{R/(r)}(t) = (1-t^d) h_R(t).
    \]
    In particular, a sequence $(r_1, r_2, \ldots, r_k)$ of degree $1$ elements of $R$ is a regular sequence if
    \[
    h_{R/(r_1, r_2, \ldots, r_k)}(t) = (1-t)^k h_R(t).
    \]
    
    By \cite[Theorem 4.4]{Jia23}, the Hilbert series of $\CC[\Ominbar]$ is given by
    \[
    h_{\CC[\Ominbar]}(t) = \frac{\sum \limits_{k=0}^n \binom{n}{k}^2 t^k}{(1-t)^{2n}}.
    \]
    Comparing with the Hilbert series of $\CC[\Ominnn]$ in Theorem \ref{Thm:degeneration}, we see that the diagonal entries $(x_{11}, x_{22}, \ldots, x_{nn})$ (which generate the ideal of $\Lien^+ \oplus \Lien^-$ in $\CC[\slnn]$) form a regular sequence in $\CC[\Ominbar]$.  It follows (see, for example, \cite[Proposition 3.1.19, Exercise 3.6.20]{BH98}) that $\Ominbar$ is Gorenstein if and only if so is $\Ominnn$.  We are done by Theorem \ref{Cor:degeneration}.
\end{proof}

\iffalse
A simplicial complex is pure if its maximal faces have the same dimension.  Let $\Delta$ be a pure simplicial complex.  We say that $\Delta$ is shellable if its maximal faces can be enumerated
\[
F_1, F_2, \ldots, F_l
\]
in such a way that, for each $i \in [2,l]$, the intersection of $F_i$ with the simplicial subcomplex of $\Delta$ generated by $F_1, F_2, \ldots, F_{i-1}$ is generated by a nonempty set of maximal proper faces of $F_i$.

If $\Delta$ is a shellable simplicial complex, then $\CC[\Delta]$ is Cohen-Macaulay.
\fi

%%%%%%%%%%
It is evident that $\Lien^+ \oplus \Lien^- = \cup_{w \in S_{n+1}} \Ad_w \Lien^+$.  Hence, as sets,
    \[
    \Ominnn = \bigcup \limits_{w \in S_{n+1}} \Ominbar \cap \Ad_w \Lien^+ = \bigcup \limits_{w \in S_{n+1}} \Ad_w \bigl( \Ominn \bigr).
    \]
Moreover, it is easy to verify that, as a set, $\Ominn$ is a union of irreducible components of $\Ominnn$.  Therefore, it is important to understand the geometry of $\Ominn$.  It turns out that similar techniques to the proof of Theorem \ref{Thm:degeneration} and Theorem \ref{Cor:degeneration} can be applied to study $\Ominn$.  To this end, fix a presentation $\C[\Lien^+] = \CC[x_1, x_2, \ldots, x_{n(n+1)/2}]$, where we label the $(i,j)$-th coordinate (with $i < j$) by $$x_{(2n-j+i+2)(j-i-1)/2 + i},$$ that is
\begin{equation*}
\setlength\arraycolsep{0.5em}\def\arraystretch{1.5}
\begin{pmatrix}
        0 & x_1 & x_{n+1} & x_{2n} & \cdots & x_{n(n+1)/2-5} & x_{n(n+1)/2-2} & x_{n(n+1)/2} \\
        0 & 0 & x_2 & x_{n+2} & \cdots & x_{n(n+1)/2-8} & x_{n(n+1)/2-4} & x_{n(n+1)/2-1} \\
        0 & 0 & 0 & x_3 & \cdots & x_{n(n+1)/2-12} & x_{n(n+1)/2-7} & x_{n(n+1)/2-3} \\
        0 & 0 & 0 & 0 & \cdots & x_{n(n+1)/2-17} & x_{n(n+1)/2-11} & x_{n(n+1)/2-6} \\
        \vdots & \vdots & \vdots & \vdots & \ddots & \vdots & \vdots & \vdots \\
        0 & 0 & 0 & 0 & \cdots & 0 & x_{n-1} & x_{2n-1} \\
        0 & 0 & 0 & 0 & \cdots & 0 & 0 & x_n \\
        0 & 0 & 0 & 0 & \cdots & 0 & 0 & 0
    \end{pmatrix} \tag{$\varheart$}.
\end{equation*}
It is evident that the initial monomial of a two-by-two minor of ($\varheart$) with respect to the degree revlex order on $\CC[x_1, x_2, \ldots, x_{n(n+1)/2}]$ is the product of the northwest and southeast entries.

Let $Z' = \{(i,j) \colon i \geq j\} \subset [1,n+1] \times [1,n+1]$, so ${\rm Mat}_{(n+1) \times (n+1)}(Z') \cong \Lien^+$ and the ideal $I_2(Z')$ defines the affine scheme $\Ominn$.
\begin{Cor}
    A universal degree revlex Gr\"obner basis of the ideal $I_2(Z')$ of the polynomial ring $\CC[x_1, x_2, \ldots, x_{n(n+1)/2}]$ is given by
    \[
    \mathcal{G}_2(Z') = \{ \text{two-by-two minors of the matrix (} \varheart \text{)} \}.
    \]
\end{Cor}
\begin{proof}
    Apply Theorem \ref{Thm:2by2Groebner} to ${\rm Mat}_{(n+1) \times (n+1)}$ and $Z'=\{(i,j) \colon i \geq j\}$.
\end{proof}

\begin{Thm}
    The scheme-theoretic intersection $\Ominn$ is reduced and Cohen-Macaulay.
\end{Thm}

\begin{proof}
%%%%%%%%
\iffalse
Fix a total order
\begin{align*}
    \begin{pmatrix}
        0 & x_1 & x_{n+1} & x_{2n} & \cdots & x_{n(n+1)/2-5} & x_{n(n+1)/2-2} & x_{n(n+1)/2} \\
        0 & 0 & x_2 & x_{n+2} & \cdots & x_{n(n+1)/2-8} & x_{n(n+1)/2-4} & x_{n(n+1)/2-1} \\
        0 & 0 & 0 & x_3 & \cdots & x_{n(n+1)/2-12} & x_{n(n+1)/2-7} & x_{n(n+1)/2-3} \\
        0 & 0 & 0 & 0 & \cdots & x_{n(n+1)/2-17} & x_{n(n+1)/2-11} & x_{n(n+1)/2-6} \\
        \vdots & \vdots & \vdots & \vdots & \ddots & \vdots & \vdots & \vdots \\
        0 & 0 & 0 & 0 & \cdots & 0 & x_{n-1} & x_{2n-1} \\
        0 & 0 & 0 & 0 & \cdots & 0 & 0 & x_n \\
        0 & 0 & 0 & 0 & \cdots & 0 & 0 & 0
    \end{pmatrix}
\end{align*}
\fi
%%%%%%%%
Define a partial order $\prec$ on $\{(i,j) \colon 1 \leq i < j \leq n+1\}$ by declaring that $(i,j) \prec (k,l)$ if $[i,j] \subset [k,l]$.  The order complex $\Delta'$ of this poset is the abstract simplicial complex whose ground set is $\{(i,j) \colon 1 \leq i < j \leq n+1\}$ and whose faces are chains $(i_1,j_1) \prec (i_2,j_2) \prec \cdots \prec (i_l,j_l)$.  Identify $\alpha\in\Phi_+^{A_n}$ with the position $(i,j)$ of the $\alpha$-root space.  It is then clear that, with respect to the degree revlex order determined by ($\varheart$), the initial ideal of the ideal of $\Ominn$ in $\C[\Lien^+]$ is the Stanley-Reisner ideal of the order complex of $(\Delta', \prec)$, which we will prove to be shellable. 
%By definition,
%\begin{align*}
    %\CC[\Delta'] & = \CC[x_{ij} \colon 1 \leq i < j \leq n+1]/\big( x_{ij} x_{kl}, \text{ such that } \\
    %& \qquad \qquad \qquad \qquad \qquad \qquad \qquad (i,j) \text{ and } (k,l) \text{ are incomparable with respect to } \prec \big) \\
    %& = \CC[x_1, x_2, \ldots, x_{n(n+1)/2}]/\init_> \big( I_2(Z') \big),
%\end{align*}
%where $>$ stands for the degree revlex order on $\CC[x_1, x_2, \ldots, x_{n(n+1)/2}]$.
The maximal faces of the order complex are paths from $(i,i+1)$ for some $i \in [1,n]$ to $(1,n+1)$ such that each step is either east or north.  In particular, this simplicial complex is pure. A
 shelling is given as follows.  A path from $(i,i+1)$ precedes one from $(j,j+1)$ if $i < j$.  For two paths $P,Q$ from $(i,i+1)$, the path $P$ precedes $Q$ if for the first step where $P$ and $Q$ are different, $P$ is north and $Q$ is east.  

We show that this enumeration is indeed a shelling. Fix a path $P$, and let $F$ be a face of the intersection of $P$ with the simplicial subcomplex generated by the paths that precede $P$. So $P$ is not the first maximal face of the simplicial complex that contains $F$ as a face. 
In particular, $F$ does not contain some southeast corner of $P$ (we regard $(i,i+1)$ as a southeast corner of $P$ if the first step of $P$ is north). Let $Q$ be the path obtained from $P$ by replacing this southeast corner with the northwest one.  By definition, the path $Q$ precedes $P$, and $P$ and $Q$ intersect in a maximal proper face of $P$ which contains $F$. 

Therefore, by Theorem \ref{Thm:shellable}, the Stanley-Reisner ring $\CC[\Delta_{\Sigma}]$ is Cohen-Macaulay.  The rest of the proof is the same as that of Theorem \ref{Cor:degeneration}.
\end{proof}

%\nocite {*}
\bibliographystyle{alpha} % We choose the "alpha" reference style
\bibliography{refs} % Entries are in the refs.bib file

\end{document}